\documentclass[11pt,letterpaper]{amsart}

\usepackage[pdftex,colorlinks,linkcolor=blue,citecolor=red,filecolor=black]{hyperref}

\usepackage{amsmath}
\usepackage{amsthm}
\usepackage{amssymb}
\usepackage{xfrac}
\usepackage{faktor}
\usepackage{mathtools}
\usepackage{mathrsfs}
\usepackage{tikz}
\usepackage{framed}
\usepackage{upgreek}
\usepackage{bbm}
\usepackage{bm}

\usepackage{enumitem}
\setlist[enumerate]{label=({\roman*})}

\newtheorem{theorem}{Theorem}
\newtheorem{proposition}[theorem]{Proposition}

\newtheorem{lemma}[theorem]{Lemma}

\theoremstyle{definition}

\newtheorem{remark}[theorem]{Remark}

\numberwithin{equation}{section}
\numberwithin{theorem}{section}

\def\<{\langle}
\def\>{\rangle}

\def\Z{\mathbb{Z}}
\def\R{\mathbb{R}}

\def\C{\mathbb{C}}

\def\wtM{\widetilde M}
\def\wtX{\widetilde X}
\def\oG{\overline{G}}
\def\oM{\overline M}
\begin{document}
\bibliographystyle{plain} \title[Abelian Liv\v{s}ic theorem]% \lebn]
{Abelian Liv\v{s}ic theorems for Anosov flows}

\author{Richard Sharp} 
\address{Mathematics Institute, University of Warwick,
Coventry CV4 7AL, U.K.}
\email{R.J.Sharp@warwick.ac.uk}

\thanks{\copyright 2025. This work is licensed by a CC BY license.}

\keywords{}

\begin{abstract}
We give two short proofs of the abelian Liv\v{s}ic theorem of Gogolev and Rodriguez Hertz.
We show that these proofs may be extended to give new abelian Liv\v{s}ic theorems
for positive density sets of null-homologous orbits and for amenable covers.
\end{abstract}

\maketitle

%
%
%
%
%
%
%
%
%
%%%%%%%%%%%%%%%%%%%%%%%%%%%%%%%%%%%%%%%%%%%%%%%%%%%%%%%%%%%%%%%%%%%%%%%%%%%%%

\section{introduction}

Let $M$ be a compact smooth Riemannian manifold  and let $X^t : M \to M$ be a transitive 
Anosov flow,
generated by the vector field $X$. (An Anosov flow is transitive if it has a dense orbit.)
Let $\mathcal P$ denote the set of prime periodic orbits of the flow
and let $\ell(\gamma)$ denote the least period of $\gamma \in \mathcal P$.
(A periodic orbit is called prime if it traverses its image only once.)
For $f : M \to \R$ and $\gamma \in \mathcal P$, write
\[
\int_\gamma f = \int_0^{\ell(\gamma)} f(X^t(x_\gamma)) \, dt
\]
for any $x_\gamma$ on $\gamma$.

A classical result is the Liv\v{s}ic (or Livshits) periodic orbit theorem: if 
a H\"older continuous function $f : M \to \mathbb R$ satisfies
\begin{equation}\label{eq:liv}
\int_\gamma f = 0 \quad \forall \gamma \in \mathcal P
\end{equation}
then $f = L_Xu$, where $u : M \to \mathbb R$ is a H\"older continuous function
(with the same exponent as $f$) which is continuously differentiable along flow lines and $L_X$ is the Lie derivative
\cite{Liv72}. 
A more recent addition to this theory is the beautiful abelian Liv\v{s}ic theorem of Andrey 
Gogolev and Federico Rodriguez Hertz, which characterises H\"older functions $f$ for which
it is only assumed that (\ref{eq:liv}) holds for \emph{null-homologous} periodic orbits.
For $\gamma \in \mathcal P$ (or more generally for any closed curve on $M$), let 
$[\gamma] \in H_1(M,\mathbb Z)$ denote the homology class of $\gamma$. 
We say that $X^t$ is \emph{homologically full} if every class in $H_1(M,\mathbb Z)$ is represented by 
an element of $\mathcal P$.
Write $\mathcal P_0 = \{\gamma \in \mathcal P \hbox{ : } [\gamma] =0\}$.

\begin{theorem}[Gogolev and Rodriguez Hertz \cite{GRH}]\label{abelian_livsic_theorem}
Let 
$X^t : M \to M$ be a homologically full transitive Anosov flow.
If $f : M \to \mathbb R$ is H\"older continuous and satisfies 
\begin{equation}\label{eq:homology_null_periodic}
\int_\gamma f = 0 \quad \forall \gamma \in \mathcal P_0
\end{equation}
then
\[
f = \omega(X) + L_Xu,
\]
for some smooth closed $1$-form $\omega$ and some H\"older continuous $u : M \to \mathbb R$
which is continuously differentiable along flow lines.
\end{theorem}

\begin{remark} If the first Betti number of $M$ is zero, so that $H_1(M,\mathbb Z)$ is
finite, then $M$ admits no non-zero closed $1$-form and it is easy to see that Theorem
\ref{abelian_livsic_theorem} reduces to the classical Liv\v{s}ic theorem.
(There are many examples of transitive Anosov flows on manifolds with zero first Betti number. The simplest are given by geodesic flows over two-dimensional hyperbolic orbifolds of genus zero.
Such orbifolds have the form $S = \mathbb H^2/\Gamma$, where $\mathbb H^2$ is the hyperbolic plane and 
$\Gamma < \mathrm{PSL}(2,\R)$ is a Fuchsian group such that $S$ has $p$ cone points satisfying $p \ge 5$ or $p=4$ and the orders of these points are not all $2$, or $p =3$ and the sum of the reciprocals of the orders is less than $1$.
Then we can define the geodesic flow on $M = \mathrm{PSL}(2,\R)/\Gamma$, which has zero first Betti number (Lemma 2.1 of \cite{Dehornoy}), and the flow is transitive Anosov.
More examples are given by surgery, see \cite{Goodman} and \cite{FoulonHasselblatt}.
These flows are automatically homologically full since the periodic orbits are equidistributed
with respect to any finite group \cite{ParryPollicott86}, \cite{Sunada84}.)
\end{remark}

We give two new and short proofs of Theorem \ref{abelian_livsic_theorem}, one 
based on the weighted equidistribution theorems for null-homologous periodic orbits 
in \cite{CS} and the other on older asymptotic counting results of Lalley \cite{lalley},
Sharp \cite{Sharp93} and
Babillot and Ledrappier \cite{bab-led}. Our second proof also gives
an abelian Liv\v{s}ic theorems for sets of null-homologous orbits with positive density
(analogous to the positive density version of the classical Liv\v{s}ic theorem
obtained
recently by Dilsavor and Marshall Reber \cite{DilsavorReber}). 

\begin{theorem}\label{th:positive_proportion}
Let 
$X^t : M \to M$ be a homologically full transitive Anosov flow.
If, for some $\Delta>0$, a H\"older continuous function $f : M \to \mathbb R$ satisfies 
\begin{equation}\label{eq:positive_proportion_null_periodic}
\limsup_{T \to \infty}
\frac{\#\left\{\gamma \in \mathcal P_0 \hbox{ : } T<\ell(\gamma) \le T+\Delta, \ \int_\gamma f =0\right\}}
{\#\{\gamma \in \mathcal P_0 \hbox{ : } T<\ell(\gamma) \le T+\Delta\}}
>0
\end{equation}
then
\[
f = \omega(X) + L_Xu,
\]
for some smooth closed $1$-form $\omega$ and some H\"older continuous $u : M \to \mathbb R$
which is continuously differentiable along flow lines.
\end{theorem}

We also have an abelian Liv\v{s}ic theorem for amenable covers.
For the next result, $\widetilde X^t : \widetilde M \to \widetilde M$ is the lifted flow on
a regular cover of $M$. Let $G$ be the covering group, with identity element $e$.
For each $\gamma \in \mathcal P$, we can associated its \emph{Frobenius class}
$\langle \gamma \rangle$, which is a conjugacy class in $G$. If $\tilde \gamma$ is any lift 
of $\gamma \in \mathcal P$ then the initial and final endpoints of $\tilde \gamma$ are related
by the action of some $g \in \langle \gamma \rangle$, and $\tilde \gamma$ is itself a periodic orbit if and only if $\langle \gamma \rangle =\{e\}$.
Write $\widetilde{\mathcal P}_0 = \{\gamma \in \mathcal P \hbox{ : } \langle \gamma \rangle =\{e\}\}$.

\begin{theorem}\label{th:amenable}
Let 
$X^t : M \to M$ be a homologically full transitive Anosov flow and let
$\widetilde M$ be a regular cover of $M$ such that
\begin{itemize}
\item
the lifted flow $\widetilde X^t : \widetilde M \to \widetilde M$ is topologically transitive, and
\item the covering group $G$ is amenable.
\end{itemize}
If $f : M \to \mathbb R$ is H\"older continuous and satisfies 
\begin{equation}\label{eq:amenable_null_periodic}
\int_\gamma f = 0 \quad \forall \gamma \in \widetilde{\mathcal P}_0
\end{equation}
then
\[
f = \omega(X) + L_Xu,
\]
for some smooth closed $1$-form $\omega$ and some H\"older continuous $u : M \to \mathbb R$
which is continuously differentiable along flow lines.
\end{theorem}

In the next section we give some background
on Anosov flows and cohomology. In section \ref{sec:two_proofs}, we give two proofs of Theorem \ref{abelian_livsic_theorem}
(the second of which also implies Theorem \ref{th:positive_proportion}).
In section \ref{sec:general_abelian_cover}, we consider general abelian covers.
In section \ref{sec:amenable}, we discuss how to extend the abelian Liv\v{s}ic theorem to amenable covers 
and prove Theorem \ref{th:amenable}.

I am grateful to Stephen Cantrell, Caleb Dilsavor, Andrey Gogolev and Federico Rodriguez Hertz for helpful
comments on an early version of this note.

\section{Anosov flows and cohomology}\label{sec:anosov}

Let $M$ be a compact Riemannian manifold and $X^t : M \to M$ be a $C^1$ transitive
Anosov flow \cite{anosov} (for a comprehensive modern treatment, see \cite{fisher}). 
We have
$H_1(M,\mathbb Z) \cong \mathbb Z^b \oplus \mathfrak F$, where $b=b_1(M) \ge 0$ is the first Betti
number of $M$ and $\mathfrak F$ is a finite abelian group.

We say that
$X^t$ is \emph{homologically full} if the map $\mathcal P \to H_1(M,\Z) : \gamma \mapsto
[\gamma]$ is a surjection. 
(See Remark \ref{rem:examples_of_hf} for examples.)
This automatically implies that the flow is weak-mixing 
(since an Anosov flow fails to be weak-mixing only when it is a constant
suspension of an Anosov diffeomorphism \cite{plante},
it which case it can have no
null-homologous periodic orbits). 
From now on, we assume that $b \ge 1$ and ignore any torsion in $H_1(M,\mathbb Z)$
(so we interpret $[\gamma]$ as an element of $H_1(M,\mathbb Z)/\mathfrak F$).

We interpret the real cohomology group
$H^1(M,\mathbb R)$
as the de Rham cohomology group. i.e. the quotient of
the space of smooth closed $1$-forms on $M$ by the space of smooth exact $1$-forms. 
We write $[\omega]$ for the cohomology class determined by the closed $1$-form $\omega$,
i.e. if $[\omega]=[\omega']$ then $\omega-\omega'$ is an exact form, and we say that $\omega$ has integral periods if $\int_c \omega \in \mathbb Z$ for every smooth closed curve in $M$.
We can choose forms $\omega_1,\ldots,\omega_b$ with integral periods so that
$[\omega_1],\ldots,[\omega_b]$ is a basis for $H^1(M,\mathbb R)$. 
This choice determines a fixed isomorphism between $H_1(M,\mathbb Z)/\mathfrak F$ and $\mathbb Z^b$
given by
\[
[c] \mapsto \left(\int_c \omega_1,\ldots,\int_c \omega_b\right),
\]
where $c$ is a smooth closed curve on $M$ and $[c] \in H_1(M,\mathbb Z)/\mathfrak F$ is its homology class.
If $\gamma \in \mathcal P$ then we have
\[
 \left(\int_\gamma \omega_1,\ldots,\int_\gamma \omega_b\right)
=\left(\int_\gamma \omega_1(X),\ldots,\int_\gamma \omega_b(X) \right).
\]

Let $\mathcal M(X)$ denote the set of $X^t$-invariant Borel probability measures on $M$.
Given $\nu \in \mathcal M(X)$, we define the associated
winding cycle (or asymptotic cycle) $\Phi_\nu \in H_1(M,\R)$ by
\[
\langle \Phi_\nu, [\omega] \rangle = \int \omega(X) \, d\nu,
\]
where $\langle \cdot,\cdot \rangle$ is the duality pairing
(Schwartzman \cite{Sch}, Verjovsky and Vila Freyer \cite{VV}).

\begin{remark}\label{rem:examples_of_hf}
There are many examples of homologically full transitive Anosov flows.
Most obviously, geodesic flows $X^t : T^1N \to T^1N$, where $M=T^1N$ is the unit-tangent bundle over a compact 
Riemannian manifold $N$ with negative sectional curvatures. Homological fullness holds here because every non-trivial free homotopy class in $N$ contains a closed geodesic (see Theorem 2.2 in Chapter 12 of \cite{doCarmo})
and closed geodesics on $N$ are in natural one-to-one correspondence with periodic orbits for the geodesic flow, while 
$H_1(N,\Z)$ and $H_1(T^1N,\Z)$ are isomorphic if $\dim(M) \ge 3$ and satisfy $H_1(T^1N,\Z) \cong H_1(N,\Z)
\oplus \Z/(-\chi(N))\Z$, where $\chi(N)$ is the Euler characteristic of $N$.
More generally, any contact Anosov flow is transitive and homologically full.
(Recall that $X^t : M \to M$ is a contact flow if $\dim(M)=2k+1$ and there exists a $1$-form $\alpha$ on $M$ such that
the volume form $\alpha \wedge (d\alpha)^k$ is $X^t$-invariant.)
To see this, we use the characterization in the proof of Theorem 1 of \cite{Sharp93} that $X^t$ is
homologically full if and only if $\Phi_{\mu_\varphi}=0$ for some
 H\"older continuous function $\varphi : M \to \R$. Since the volume measure $m$ corresponding to $\alpha \wedge (d\alpha)^k$ is the equilibrium state of a H\"older continuous function, Corollary 4.10 of \cite{plante} gives that $\Phi_m=0$.
\end{remark}

In addition to the cohomology of the manifold $M$, we will also consider 
the dynamical cohomology associated to the flow $X^t : M \to M$.
For $0<\theta<1$, let $C^\theta(M,\mathbb R)$ denote the space of H\"older continuous functions 
from $M$ to $\mathbb R$ with exponent $\theta$.
The theorems we prove in this paper concern  H\"older continuous functions $f : M \to \R$ and, given
a particular $f$, we fix $\theta$ to be its H\"older exponent. 
A function $g \in C^\theta(M,\R)$  is called a ($C^\theta$ dynamical) coboundary if $g= L_Xu$, for some function
$u: M \to \mathbb R$ which is differentiable along flow lines,
and two functions are cohomologous if their difference is a coboundary. We write $B^\theta(M,\mathbb R)$ for the intersection of the set of 
flow coboundaries with
$C^\theta(M,\mathbb R)$, i.e. $B^\theta(M,\mathbb R)$ is the set of functions $L_Xu$ such that
 $u \in C^\theta(M,\mathbb R)$ with $u$ $C^\theta$-differentiable along flow lines.
Finally we define the dynamical cohomology group $\mathcal H_\theta^1(X,\mathbb R)$ by
$\mathcal H_\theta^1(X,\mathbb R) := C^\theta(M,\mathbb R)/B^\theta(M,\mathbb R)$.
Given a H\"older continuous function $g \in C^\theta(M,\R)$, we write $[g]$ for its 
($C^\theta$ dynamical) cohomology class (where the notation suppresses the dependence on $\theta$).
We can define a linear map $\iota : H^1(M,\R) \to \mathcal H_\theta^1(X,\mathbb R)$ as follows.
Given $w \in H^1(M,\mathbb R)$, choose a smooth closed $1$-form $\omega$ in this class. Then define
$\iota(w) = [\omega(X)]$, i.e. the ($C^\theta$ dynamical) cohomology class of the function $\omega(X) \in C^\theta(M,\mathbb R)$.

\begin{lemma}\label{lem:iota_injective}
The map $\iota : H^1(M,\R) \to \mathcal H_\theta^1(X,\mathbb R)$ is an injection.
\end{lemma}

\begin{proof}
Suppose that $\iota([\omega])=0$. Then 
\[
\langle [\gamma],[\omega] \rangle= \int_\gamma \omega = \int_\gamma \omega(X) =0
\]
for every $\gamma \in \mathcal P$. Since $\{[\gamma] \hbox{ : } \gamma \in \mathcal P\}$ generates $H_1(M,\Z)$ 
as a group (this is true even without the assumption that the flow is homologically full \cite{ParryPollicott86}, \cite{Sunada84}), we can conclude that $[\omega]=0$, so $\iota$ is an injection.
\end{proof}

We will also use the Bruschlinsky theory of cohomology \cite{Bruschlinsky}.
Let $K = \{z \in \C \hbox{ : } |z|=1\}$ denote the unit circle. Then we can interpret
$H^1(M,\Z)$ as the set of continuous functions $u : M \to K$ modulo functions homotopic to the identity
(i.e. modulo functions of the form $e^{2\pi i h(x)}$, where $h \in C(M,\R)$), and we let $[u]$ be the cohomology 
class represented by $u$.
Furthermore, it is shown in \cite{Sch} that we can assume the representative $u$ is continuously differentiable along 
flow lines, in which case we have
\[
\langle [\gamma],[u]\rangle = \frac{1}{2\pi i} \int_\gamma \frac{L_Xu}{u}
\]
for all $\gamma\in \mathcal P$.
We will use the following lemma, which appeared as Proposition 3.7 in \cite{GRH}.

\begin{lemma}\label{lem:discrete_periods_imply_cohomological}
Let $f : M \to \mathbb R$ be a H\"older continuous function with H\"older exponent $\theta$ such that
the set of $f$-periods $\left\{ \int_\gamma f \hbox{ : } \gamma \in 
\mathcal P\right\}$ is contained in a discrete subgroup of $\R$. Then 
$[f]\in \iota(H^1(M,\R))$, i.e. $f$ is cohomologous (in $C^\theta(M,\mathbb R)$)
to a function of the form $\omega(X)$, for some 
smooth closed $1$-form $\omega$.
\end{lemma}

\begin{proof}
Suppose that the set of $f$-periods is contained in $c\Z$, for some $c>0$. Then $e^{2\pi i\int_\gamma c^{-1}f}=1$ for all
$\gamma \in \mathcal P$. Applying the Liv\v{s}ic theorem for $K$, there is a function $u : M \to K$,
$C^1$ along flow lines, such that
\[
u(X^tx) = u(x) e^{2\pi i\int_0^t c^{-1}f(X^sx) \, ds},
\]
for all $x \in M$. Hence
\[
c^{-1}f = \frac{1}{2\pi i} \frac{L_Xu}{u}.
\]
Now choose a smooth closed $1$-form $\eta$ in the cohomology class $[u] \in H^1(M,\mathbb Z)$.
For each $\gamma \in \mathcal P$, we then have
\[
\int_\gamma c^{-1} f = \frac{1}{2\pi i} \int_\gamma \frac{L_X}{u}
= \langle [\gamma],[u] \rangle = \int_\gamma \eta = \int_\gamma \eta(X).
\]
It follows from the classical Liv\v{s}ic theorem that $f$ is cohomologous $\omega(X)$, where $\omega=c\eta$.
\end{proof}

Let $\varphi : M \to \mathbb R$ be H\"older continuous. We define its pressure $P(\varphi)$ by
\[
P(\varphi) = \sup\left\{h(\nu) + \int \varphi \, d\nu \hbox{ : } \nu \in \mathcal M(X)\right\}
\]
and there is a unique $\mu_\varphi \in \mathcal M(X)$, called the equilibrium state for $\varphi$,
at which the supremum is attained. The following result is fundamental.

\begin{lemma}\label{lem:fundamental}
If $\varphi,\psi : M \to \mathbb R$ are H\"older continuous then $\mu_\varphi =\mu_\psi$
if and only if $\varphi-\psi$ is a cohomologous to a constant.
\end{lemma}

\begin{proof}
This is a standard result but it is hard to find a clear reference. First, we note that
the ``if'' direction is trivial.
For the other direction, we can proceed using symbolic dynamics.
Suppose $\mu_\varphi =\mu_\psi$. Without loss of generality, we can add constants to 
$\varphi$ and $\psi$ so that $P(\varphi)=P(\psi)=0$.
 By the classical results of Bowen
\cite{Bowen} and Ratner \cite{Ratner}, we can find a suspension flow $\sigma^t : \Sigma^r \to \Sigma^r$
over a mixing subshift of finite type $\sigma : \Sigma \to \Sigma$, where $r : \Sigma \to \mathbb R^{>0}$
is H\"older continuous, and a H\"older continuous surjection $\pi : \Sigma^r \to M$ that is one-to-one 
on a residual set and satisfies $X^t \circ \pi = \pi \circ \sigma^t$. 
Furthermore, $\pi^*(\mu_\varphi)
= \mu_{\varphi \circ \pi}$ and $\pi^*(\mu_\psi)=\mu_{\psi \circ \pi}$
(and $\pi$ is a.e. one-to-one between the respective pairs of measures), where the measures on the right 
are equilibrium states with respect to the suspension flow, and $P(\varphi \circ \pi)=P(\psi\circ \pi)=0$. 
Thus, if $\mu_\varphi =\mu_\psi$
then $\mu_{\varphi \circ \pi}=\mu_{\psi \circ \pi}$. 
 If we define
$\Phi, \Psi : \Sigma \to \mathbb R$ by
\[
\Phi(x) = \int_0^{r(x)} \varphi \circ \pi(\sigma^t(x,0)) \, dt
\mbox{ and }
\Psi(x) = \int_0^{r(x)} \psi \circ \pi(\sigma^t(x,0)) \, dt
\]
then (by Proposition 6.1 of \cite{PP}) $P(\Phi)=P(\Psi)=0$ and
\[
\mu_{\varphi \circ \pi} = \frac{m_\Phi \times \mathrm{Leb}}{\int r \, dm_\Phi}
\mbox{ and } 
\mu_{\psi \circ \pi} = \frac{m_\Psi \times \mathrm{Leb}}{\int r \, dm_\Psi},
\]
where $m_\Phi$ and $m_\Psi$ are the equilibrium states of $\Phi$ and $\Psi$, respectively, with respect to $\sigma : \Sigma \to \Sigma$.
We can then deduce that $m_\Phi=m_{\Psi}$. By Proposition 3.6 of \cite{PP}, $\Phi$ and $\Psi$ are 
cohomologous. (Here, we have use that $P(\Phi)=P(\Psi)=0$ to ensure the additive constant is zero.)
We can then deduce that $\Phi$ and $\Psi$ have the same sums around each $\sigma$-periodic orbit
and hence that $\int_\gamma \varphi - \int_\gamma \psi =0$ for all $\gamma \in \mathcal P$.
Applying, Liv\v{s}ic's theorem, $\varphi$ and $\psi$ are cohomologous.
\end{proof}

\section{Two proofs of Theorem \ref{abelian_livsic_theorem}}\label{sec:two_proofs}

We will now prove Theorem \ref{abelian_livsic_theorem}. 

\begin{proof}[First proof of Theorem \ref{abelian_livsic_theorem}]
Following section 5 of \cite{CS}, given a H\"older continuous function
$\varphi : M \to \mathbb R$, we can find a unique $ \xi(\varphi) \in \mathbb R^b$ such that
the equilibrium state of $\varphi+ \sum_{i=1}^b \xi_i(\varphi) \omega_i(X)$,
which we shall denote by $\mu(\varphi)$,
satisfies $\Phi_{\mu(\varphi)}=0$ and
\[
h(\mu(\varphi)) = \int \varphi \, d\mu(\varphi) = \sup\left\{
h(\nu) + \int \varphi \, d\nu \hbox{ : } \nu \in \mathcal M_X \mbox{ and } \Phi_\nu=0\right\}. 
\]
Furthermore, from Theorem 6.7 of \cite{CS}, we see that $\mu$ is a weak$^*$ limit
 of averages
 of null-homologous orbital measures. More precisely, for all continuous functions
 $\psi : M \to \mathbb R$, we have
 \begin{equation}\label{eq:coles-sharp}
 \lim_{T \to \infty} 
 \left(\sum_{\substack{\gamma \in \mathcal P_0 \\ T < \ell(\gamma) \le T+1}}
 e^{\int_\gamma \varphi}\right)^{-1}
 \sum_{\substack{\gamma \in \mathcal P_0 \\ T < \ell(\gamma) \le T+1}} 
 e^{\int_\gamma \varphi}
 \frac{ \int_\gamma\psi }{\ell(\gamma)}
 = \int \psi \, d\mu(\varphi).
 \end{equation}
 
 Now let $f : M \to \mathbb R$ be a H\"older continuous function satisfying $\int_\gamma f =0$ for all $\gamma \in \mathcal P_0$.
If we compare the cases $\varphi=0$ and $\varphi=f$ then the corresponding terms on the left hand side 
of 
(\ref{eq:coles-sharp}) are equal
 and so
we obtain that $\mu(0)=\mu(f)$.
Applying Lemma \ref{lem:fundamental}, we see that
\[
f = L_Xu + \sum_{i=1}^b (\xi_i(0)-\xi_i(f))\omega_i(X) +c,
\]
where $u : M \to \mathbb R$ is a H\"older function which continuously differentiable along flow lines
and $c \in \mathbb R$. We can see from (\ref{eq:coles-sharp}) (with $\psi=f$)
that $\int f \, d\mu(0)=0$ and, combined with $\Phi_{\mu(0)}=0$, this gives $c =0$, completing the proof.
\end{proof}

\begin{proof}[Second proof of Theorem \ref{abelian_livsic_theorem}]
An alternative, slightly longer, argument is to show that the failure of Theorem \ref{abelian_livsic_theorem}
is inconsistent with the periodic orbit counting results of \cite{bab-led}, \cite{lalley}, \cite{Sharp93}.
We fix a H\"older continuous function $f : M \to \mathbb R$ satisfying (\ref{eq:homology_null_periodic}). Then, clearly,
\[
\#\{\gamma \in \mathcal P_0 \hbox{ : } T<\ell(\gamma) \le T+1\}
= \#\left\{\gamma \in \mathcal P_0 \hbox{ : } T<\ell(\gamma) \le T+1, \  \int_\gamma f=0\right\}.
\]
Now, we have from \cite{Sharp93} that
\begin{equation}\label{eq:to_be_shown_impossible}
\#\{\gamma \in \mathcal P_0 \hbox{ : } T<\ell(\gamma) \le T+1\}
\sim C\frac{e^{\alpha T}}{T^{1+b/2}}, \qquad \mbox{as } T \to \infty,
\end{equation}
for some $C>0$ and $\alpha>0$. We aim to show that if $f$ is not cohomologous to $\omega(X)$, for some closed $1$-form $\omega$ (i.e. that the conclusion of Theorem \ref{abelian_livsic_theorem} fails to hold), then the asymptotic (\ref{eq:to_be_shown_impossible}) cannot be true, giving a contradiction.

Suppose that  $f$ is not cohomologous to $\omega(X)$, for some closed $1$-form $\omega$.
Then, in particular, by Lemma \ref{lem:discrete_periods_imply_cohomological}, the set of $f$-periods
$\{\int_\gamma f \hbox{ } \gamma \in \mathcal P\}$ is not contained in a discrete subgroup of $\R$.
We will use Theorem 1.2 of \cite{bab-led} and define a function $F : M \to \R^{b+1}$ by 
$F = (f,\omega_1(X),\ldots,\omega_b(X))$, so that 
\[
\int_\gamma F = \left(\int_\gamma f, [\gamma]\right),
\]
and define
\[
\mathcal C(F)= \left\{\int F \, d\nu \hbox{ : } \nu \in \mathcal(X)\right\} \subset \R^{b+1}.
\]
Let $\Gamma$ denote the smallest closed subgroup of $\R^{b+1}$ 
containing
$\{\int_\gamma F \hbox{ : } \gamma \in \mathcal P\}$;
clearly, $\Gamma = \R \times \Z^b$.
We also let $\widetilde \Gamma$ denote the smallest closed subgroup of $\R^{b+2}$ 
containing
$
\{(\ell(\gamma),\int_\gamma F) \hbox{ : } \gamma \in \mathcal P\}$.

Let $g_0 : \R \to \R^+$ be the indicator function of the interval $(-1,0]$ and let $g : \Gamma \to \R^+$ be defined by
$g(y_0,y_1,\ldots,y_b) = \kappa(y_0)\delta_0(y_1,\ldots,y_b)$, where $\delta_0 :\Z^b \to \R^+$ is the indicator 
function of $\{0\}$ and $\kappa : \R \to \R^+$ is any continuous compactly supported function satisfying $\kappa(0)=1$.
Then
\[
\#\left\{\gamma \in \mathcal P_0 \hbox{ : } T<\ell(\gamma) \le T+1, \  \int_\gamma f=0\right\}
= \sum_{\gamma \in \mathcal P} g_0(\ell(\gamma)-T) \, g\left(\int_\gamma F\right).
\]
The right hand side is of the form considered in Theorem 1.2 of \cite{bab-led}
(the theorem requires $g_0$ to be continuous but an easy approximation argument allows one to pass from continuous functions
to the indicator function of an interval).
Thus we will have that
\begin{equation}\label{eq:the_other_asymptotic}
\sum_{\gamma \in \mathcal P} g_0(\ell(\gamma)-T) \, g\left(\int_\gamma F\right)
\sim 
C'\frac{e^{\alpha' T}}{T^{1+(b+1)/2}}, \qquad \mbox{as } T \to \infty,
\end{equation}
for some $0<\alpha' \le \alpha$ and  $C'>0$, provided we can show that $0 \in \mathrm{int}(\mathcal C(F))$ and that 
$\widetilde \Gamma = \R \times \Gamma$
(Assumption (A) in \cite{bab-led}). (Actually, one can show that $\alpha=\alpha'$ but this is not required for the argument.)

\medskip
\noindent
\emph{Justification of $0 \in \mathrm{int}(\mathcal C(F))$.} Suppose that $\langle a,F \rangle = a_0f+ \sum_{i=1}^b a_i\omega_i(X)$ is cohomologous
to a constant $c$ for some $a = (a_0,a_1,\ldots,a_b) \in \R^{b+1}$. As in the first proof above, $\int f \, d\mu(0)=0$
and $\Phi_{\mu(0)}=0$, so that $\int F \, d\mu(0)=0$ and $c=0$. If $a_0 \ne 0$, then $f$ is cohomologous to
$a_0^{-1}(\sum_{i=1}^b a_i\omega_i(X))$, which is ruled out by our initial assumption. So $a_0=0$ and hence
\[
\langle (a_1,\ldots,a_b), [\gamma]\rangle =0 
\quad \forall \gamma \in \mathcal P.
\]
Since $\{[\gamma] \hbox { : } \gamma \in \mathcal P\}$ generates $H_1(M,\Z)$, we must have $(a_1,\ldots,a_b)=0$
and hence $a=0$.
If we define $\mathfrak p_F : \R^{b+1}\to \R$ by $\mathfrak p_F(w) = P(\langle w,F \rangle)$,
then it is a standard result that
 $\nabla \mathfrak p_F$ is a diffeomorphism from $\R^{b+1}$ to its image $\mathrm{Im}(\nabla \mathfrak p_F)$
and that 
\[
\nabla \mathfrak p_F(w) = \int F \, d\mu^w,
\]
where $\mu^w$ is the equilibrium state for $\langle w,F \rangle$ (see page 164 of \cite{lalley}).
Then
\[
\mathrm{Im}(\nabla \mathfrak p_F)=\left\{\int F \, d\mu^w \hbox{ : } w \in \R^{b+1}\right\} 
\]
is an open subset of $\mathcal C(F)$, and hence is contained in its interior. 
The claimed result then follows from
\[
\nabla \mathfrak p_F((0,\xi_1(0),\ldots,\xi_b(0))) = \int F \, d\mu(0) =0.
\]

\medskip
\noindent
\emph{Justification of $\widetilde \Gamma = \R \times \Gamma$.}
Clearly, $\widetilde \Gamma$ is a closed subgroup of $\R \times \Gamma$. If it is a proper subgroup, then
there will be a non-trivial character of $\R \times \Gamma$ which constantly takes the value $1$ on $\widetilde \Gamma$.
So, to prove the claim, we will show that no such non-trivial character can exist.

A character of $\R \times \Gamma$ has the form
$
\Xi_{\alpha,a}(\ell,x_0,x_1,\ldots,x_b)
:=e^{2\pi i (\alpha \ell +\sum_{j=0}^b a_jx_j)}
$,
with $\alpha \in \R$ and $a=(a_0,a_1,\ldots,a_b) \in \R \times (\R/\Z)^b$. If $\Xi_{\alpha,a}$ is trivial on $\widetilde \Gamma$ then
\begin{equation}\label{eq:trivial_on_wtG}
\Xi_{\alpha,a}\left(\ell(\gamma),\int_\gamma f,\int_\gamma \omega_1(X),\ldots,\int_\gamma 
\omega_b(X) \right)
=1 \quad \forall \gamma \in \mathcal P
\end{equation}
For brevity, we will write 
$\omega = \sum_{i=1}^b a_i\omega_i$.
Let $K = \{z \in \mathbb C \hbox{ : } |z|=1\}$.
As in the proof of Lemma \ref{lem:discrete_periods_imply_cohomological}, applying the Liv\v{s}ic theorem for the circle
$K$ gives us a function $u :M \to K$, $C^1$ along flow lines,
such that
\[
u(X^tx) = u(x) e^{2 \pi i\alpha t + 2\pi i\int_0^t (a_0f(X^sx) +\omega(X)(X^sx)) \, ds},
\]
for all $x \in M$, and hence
\[
\frac{1}{2\pi i} \frac{L_Xu}{u} = \alpha  +a_0f + \omega(X).
\]
Since $f$ and $\omega(X)$ integrate to zero with respect to $\mu(0)$,
 we obtain
\[
\alpha = \int \frac{1}{2\pi i} \frac{L_Xu}{u} \, d\mu(0).
\]
However, as in the proof of the lemma,
$u$ also represents a Bruschlinsky cohomology class $[u]$ in $H^1(M,\mathbb Z)$
(\cite{Bruschlinsky}, \cite{Sch}) and 
%the pairing between an asymptotic cycle and a cohomology class
%is given by
we have
\[
0=\langle \Phi_{\mu(0)},[u]\rangle = \int \frac{1}{2\pi i} \frac{L_Xu}{u} \, d\mu(0),
\]
so that $\alpha=0$. 
Hence, (\ref{eq:trivial_on_wtG}) reduces to
\[
\exp\left(2\pi i \left(a_0\int_\gamma f +\langle (a_1,\ldots, a_b), [\gamma]\rangle\right)\right)=1 
\quad \forall \gamma \in \mathcal P,
\]
which, by the definition of $\Gamma$, implies that $a=0$ also. So $\Xi_{\alpha,a}=\Xi_{0,0}$, the trivial character
and hence $\widetilde \Gamma = \R \times \Gamma$.

We have now shown (\ref{eq:the_other_asymptotic}), which contradicts (\ref{eq:to_be_shown_impossible}).
\end{proof}

\begin{remark}
In the justification that $0$ is in the interior of $\mathcal C(F)$, one might notice that, for any $a \in \R$, one has
\[
\nabla \mathfrak p_F((a,\xi_1(0),\ldots,\xi_b(0))) = \int F \, d\mu(af) =0,
\]
which contradicts results on the strict convexity of pressure. This leads to a third proof of Theorem 
\ref{abelian_livsic_theorem}.
\end{remark}

One sees that the second proof leads to a proof of Theorem
\ref{th:positive_proportion}. (I am grateful to Andrey Gogolev for this observation.)

\begin{proof}[Proof of Theorem \ref{th:positive_proportion}]
Suppose that $f : M \to \R$ is a H\"older continuous function satisfying the condition
(\ref{eq:positive_proportion_null_periodic}). Under this hypothesis, we have
\[
\#\left\{\gamma \in \mathcal P_0 \hbox{ : } T<\ell(\gamma) \le T+1, \  \int_\gamma f=0\right\}
\le \sum_{\gamma \in \mathcal P} g_0(\ell(\gamma)-T) \, g\left(\int_\gamma F\right),
\]
where $g_0$ and $g$ are as in the second proof of Theorem \ref{abelian_livsic_theorem}.
For a contradiction, suppose that $f$ 
s not cohomologous to $\omega(X)$, for some closed $1$-form $\omega$.
If $\int f \, d\mu(0) =0$ then we can use the arguments above to show that
$
\#\{\gamma \in \mathcal P_0 \hbox{ : } T<\ell(\gamma) \le T+\Delta, \ \int_\gamma f =0\}
=o(\#\{\gamma \in \mathcal P_0 \hbox{ : } T<\ell(\gamma) \le T+\Delta\}),
$
contradicting (\ref{eq:positive_proportion_null_periodic}).
On the other hand, if $\int f \, d\mu(0) \ne 0$ then we use a large deviations argument to show that, for any $\epsilon>0$,
\[
\lim_{T \to \infty}
\frac{\#\left\{\gamma \in \mathcal P_0 \hbox{ : } T<\ell(\gamma) \le T+\Delta, \ \left|\frac{1}{\ell(\gamma)}\int_\gamma f -\int f\, d\mu(0)\right|\ge \epsilon\right\}}
{\#\{\gamma \in \mathcal P_0 \hbox{ : } T<\ell(\gamma) \le T+\Delta\}}
=0,
\]
which, if $\epsilon$ is chosen sufficiently small, also contradicts (\ref{eq:positive_proportion_null_periodic}).
\end{proof}

\begin{remark}\label{rem:weighted_positive_proportion}
In \cite{DilsavorReber}, the authors also establish a weighted result where, for a H\"older continuous function 
$\varphi : M \to \R$, one replaces counting with summing the terms $\exp (\int_\gamma \varphi)$. This type of result also holds in our setting: writing $\mathcal P_0(T,\Delta)=\{\gamma \in \mathcal P_0 : T <\ell(\gamma) \le T+\Delta\}$,
if
\[
\limsup_{T \to \infty}
\frac{\sum_{\gamma \in \mathcal P_0(T,\Delta), \, \int_\gamma f=0} \exp(\int_\gamma \varphi)}
{\sum_{\gamma \in \mathcal P_0(T,\Delta)} \exp(\int_\gamma \varphi)}
>0
\]
then
\[
f = \omega(X) + L_Xu,
\]
for some smooth closed $1$-form $\omega$ and some H\"older continuous $u : M \to \mathbb R$
which is continuously differentiable along flow lines. Following the model of the second proof of Theorem 
\ref{abelian_livsic_theorem}, this can be shown using arguments from the proof of the weighted counting result
(Theorem 6.1) in \cite{CS}.
\end{remark}

\section{General abelian covers}\label{sec:general_abelian_cover}
Let $\oM$ be any regular abelian cover of $M$, with covering group $A$ (of rank $k \ge 1$).
For simplicity of exposition, we assume that $A$ is torsion-free.
The homology class $[c]$ of a closed curve on $M$ projects to an element $[c]_A\in A$.
We say that $X^t : M \to M$ is $A$-full if the map $\mathcal P \to A : \gamma \mapsto [\gamma]_A$ is a surjection.
If $X^t$ is homologically full then
(since $[\gamma]=0$ implies $[\gamma]_A =0$) it is trivial that the conclusion of Theorem 
\ref{abelian_livsic_theorem} holds if $f$ integrates to zero over every periodic orbit 
satisfying $[\gamma]_A=0$.
However, if the flow is $A$-full (but not homologically full) then the situation is slightly more subtle.

To discuss this further, we consider the vector spaces $A \otimes \R$ (which corresponds to a quotient 
of $H_1(M,\R)$) and $(A \otimes \R)^*$ (which corresponds to a subspace of $H^1(M,\R)$). To 
each $\nu \in \mathcal M(X)$, we can associate a \emph{relative} winding cycle
$\Phi_\nu^A \in A \otimes \R$, defined by
\[
\langle \Phi_\nu^A,[\omega]\rangle = \int \omega(X) \, d\nu,
\]
where $\omega$ is a closed $1$-form with $[\omega] \in (A \otimes \R)^*$. If $X^t$ is $A$-full then we can find equilibrium states $\mu(\varphi)$ as above with $\Phi_{\mu(\varphi)}^A=0$.

Writing  $\mathcal P_0^A =\{\gamma \in \mathcal P \hbox{ : } [\gamma]_A=0\}$, we have the following theorem.

\begin{theorem}\label{th:general_abelian_covers}
Let 
$X^t : M \to M$ be an $A$-full transitive Anosov flow.
If $f : M \to \mathbb R$ is H\"older continuous and satisfies 
\begin{equation}\label{null_periodic_A}
\int_\gamma f = 0 \quad \forall \gamma \in \mathcal P_0^A
\end{equation}
then
\[
f = \omega(X) + L_Xu,
\]
for some smooth closed $1$-form $\omega$ in $(A \otimes \R)^*$
and some H\"older continuous $u : M \to \mathbb R$
which is continuously differentiable along flow lines.
\end{theorem}

\begin{proof}
The first proof of Theorem \ref{abelian_livsic_theorem} will apply here provided the weighted equidistribution result
(\ref{eq:coles-sharp}) holds with $\mathcal P_0$ replaced by $\mathcal P_0^A$.
If we look at the proof of the weighted  equidistribution result  in \cite{CS} (Theorem 6.7), we see that 
(rather than a detailed asymptotic)
we only require the exponential growth rate result given in 
Corollary 6.2 of \cite{CS}, which tells us that
\begin{equation}\label{eq:growth_rate}
\lim_{T \to \infty} \frac{1}{T} \log \sum_{\substack{\gamma \in \mathcal P_0^A \\ T < \ell(\gamma) \le T+1}}
 e^{\int_\gamma \varphi} = \beta(\varphi) :=P\left(\varphi + \sum_{i=1}^k \xi_i(\varphi) \omega_i(X)\right),
\end{equation}
where now $k = \dim (A \otimes \R)^*$, $[\omega_1],\ldots,[\omega_k]$ form an integral basis for $(A \otimes \R)^*$ and 
the $\xi_i(\varphi)$ minimizes the pressure function on the right hand side.
%This statement only requires the flow to be $A$-full.

We outline proof of (\ref{eq:growth_rate}). Suppose first that $\beta(\varphi)>0$. Following the analysis in \cite{Sharp93}, we see that $\beta(\varphi)$ is the abscissa of convergence of the Dirichlet series
\[
\sum_{\gamma \in \mathcal P_0^A} \ell(\gamma)^q e^{\int_\gamma \varphi -s\ell(\gamma)},
\]
where $q=[k/2]$, and hence 
\[
\beta(\varphi) = \limsup_{T \to \infty} \frac{1}{T} \log \sum_{\substack{\gamma \in \mathcal P_0^A
\\ \ell(\gamma) \le T}} \ell(\gamma)^q e^{\int_\gamma \varphi}.
\]
Since $\beta(\varphi)>0$, we can replace $\ell(\gamma) \le T$ with $T<\ell(\gamma) \le T+1$ and remove the 
polynomial term $\ell(\gamma)^q$, i.e. we have
\[
\beta(\varphi) = \limsup_{T \to \infty} \frac{1}{T} \log \sum_{\substack{\gamma \in \mathcal P_0^A
\\ T< \ell(\gamma) \le T+1}}  e^{\int_\gamma \varphi}.
\]
 Furthermore, following the arguments in \cite{Kempton}, one can show that the limsup is a limit. Finally, if $\beta(\varphi) \le 0$, choose $c>0$ such that 
$\beta(\varphi+c) = \beta(\varphi)+c >0$; we can then deduce the result for $\varphi$ from the result for $\varphi+c$.
\end{proof}

\begin{remark}
It is instruction to consider why the proof in the preceding section do not immediately apply.
If we wished to use the first proof of Theorem \ref{abelian_livsic_theorem} without further argument, then we would need 
to be able to apply Theorem 6.1 of \cite{CS} in the $A$-full setting.
The key step in proving this would be to show that 
$\widetilde{\Gamma_A} = \R \times \Gamma_A = \R \times \Z^k$,
where
$\Gamma_A$ is the smallest closed subgroup of $\R^k$ containing $\{[\gamma]_A \hbox{ : } \gamma \in \mathcal P\}$
(which is equal to $\Z^k$ by hypothesis) and
$\widetilde{\Gamma_A}$ is the smallest closed subgroup of $\R^{k+1}$ containing
$\{(\ell(\gamma),[\gamma]_A) \hbox{ : } \gamma \in \mathcal P\}$.
Similarly, if we wanted to use the second proof, we would need to show that
$\widetilde{\Gamma_{A,f}} = \R \times \Gamma_{A,f}$,
where
$\Gamma_{A,f}$ is the smallest closed subgroup of $\R^k$ containing $\{(\int_\gamma f,[\gamma]_A )\hbox{ : } \gamma \in \mathcal P\}$
and
$\widetilde{\Gamma_{A,f}}$ is the smallest closed subgroup of $\R^{k+1}$ containing
$\{(\ell(\gamma),\int_\gamma f, [\gamma]_A) \hbox{ : } \gamma \in \mathcal P\}$.

Let us try to show $\widetilde{\Gamma_A} = \R \times \Gamma_A$.
If we follow the argument that $\widetilde \Gamma= \R \times \Gamma$ above, we obtain the equation
\begin{equation}\label{eq:gen_ab_eq}
\frac{1}{2\pi i} \frac{L_Xu}{u} = \alpha + \omega(X),
\end{equation}
for some closed $1$-form $\omega$ with $[\omega] \in (A \otimes \R)^*$ and some $u:M \to K$
which is $C^1$ along flow lines. To proceed , one needs to show that $\alpha=0$.
If $X^t$ is homologically full then we can 
integrate with respect to a measure $\nu$ satisfying $\Phi_\nu=0$ 
 to show this and thus obtain the desired conclusion.
However, if we only assume $A$-full then integrating with respect to a measure $\nu$ satisfying $\Phi_\nu^A=0$ will be enough to kill the $\omega(X)$ term (which corresponds to a cohomology class in $(A \otimes \R)^*$)
but not enough to kill the $u$ term (which corresponds to a general cohomology class).

Nevertheless, we can show that $\alpha=0$ for large classes of Anosov flows.
First, we note that we will have $\alpha=0$ is there is \emph{any} $\nu \in \mathcal M(X)$
for which $\Phi_\nu=0$, since then integrating with respect to $\nu$ will kill both cohomological terms.
However, the existence of such a measure is equivalent not having a global cross-section
(see Proposition 7 of \cite{Sharp93} and section 7 of \cite{Sch}). Thus,
the first proof of Theorem \ref{abelian_livsic_theorem} holds
if the flow $X^t : M \to M$ does 
\emph{not} have a global cross-section.
If the flow has a global cross-section then (up to a velocity change) it is the suspension of an Anosov diffeomorphism.
We observe that the only currently known examples of Anosov diffeomorphisms are topologically conjugate
to algebraic examples on infranilmanifolds (with tori as special cases) and the induced action on homology 
is hyperbolic. In consequence, in these examples, $H_1(M,\Z)$ has rank $1$ and so there are no infinite abelian covers to consider apart from the universal homology cover.

From another point of view, passing to symbolic dynamics and modelling the flow 
by a suspension flow over a
subshift of finite type, we can use (\ref{eq:gen_ab_eq}) to conclude that if $\alpha \ne 0$ then the roof function 
is cohomologous to a locally constant function.
However, the condition that the roof function 
is \emph{not} cohomologous to a locally constant function. is 
generic \cite{FMT} and, for example, is always satisfied when $X^t$ is a co-dimension one flow or when $X^t$ is exponentially mixing.
\end{remark}

\section{Amenable covers}\label{sec:amenable}

Let $\widetilde M$ be a regular cover of $M$ such that
$G = \pi_1(M)/\pi_1(\widetilde M)$ is amenable.
Let $\wtX^t : \wtM \to \wtM$ be the lift of $X^t : M \to M$ to the cover.
We then have a weighted equidistribution theorem, which generalises 
Theorem 9.2 in \cite{Dougall-Sharp1} (which covers the case $\varphi=0$).

\begin{theorem}\label{th:eq-weighted-amenable}
Let 
$X^t : M \to M$ be a 
transitive Anosov flow and let
$\widetilde M$ be a regular cover of $M$ such that
\begin{itemize}
\item
the lifted flow $\widetilde X^t : \widetilde M \to \widetilde M$ is topologically transitive, and
\item the covering group $G$ is amenable.
\end{itemize}
Then for sufficiently large
$\Delta >0$ and any H\"older continuous function $\varphi : M \to \mathbb R$, we have
\[
 \lim_{T \to \infty} 
 \left(\sum_{\substack{\gamma \in \widetilde{\mathcal P}_0 \\ T < \ell(\gamma) \le T+\Delta}}
 e^{\int_\gamma \varphi}\right)^{-1}
 \sum_{\substack{\gamma \in \widetilde{\mathcal P}_0 \\ T < \ell(\gamma) \le T+\Delta}} 
 e^{\int_\gamma \varphi}
\,  \frac{ \int_\gamma\psi}{\ell(\gamma)} 
 = \int \psi \, d\mu(\varphi),
 \]
 for all continuous functions $\psi : M \to \mathbb R$. 
\end{theorem}

We note that the hypothesis 
$\widetilde X^t : \widetilde M \to \widetilde M$ is topologically transitive
implies that the maximal abelian subcover of $\widetilde{M}$ is full.
Using the large deviations approach of section 9 of \cite{Dougall-Sharp1}, 
Theorem \ref{th:eq-weighted-amenable} follows once one has shown the following.

\begin{proposition} \label{prop:transfer_to_abelian}
For sufficiently large $\Delta>0$,
\[
\lim_{T \to \infty} \frac{1}{T} \log
\sum_{\substack{\gamma \in \widetilde{\mathcal P}_0 \\ T < \ell(\gamma) \le T+\Delta}}
e^{\int_\gamma \varphi} =
P\left(\varphi + \sum_{i=1}^b \xi_i(\varphi)\omega_i(X)\right).
\]
\end{proposition}

Before we prove this, we need to introduce some more ideas from thermodynamic
formalism, particularly that of Gurevi\v{c} pressure. Recall the notation introduced in the 
proof of Lemma \ref{lem:fundamental}. The flow $\wtX^t : \wtM \to \wtM$ may be modelled
by a suspension flow $\tilde \sigma^t$ over a skew-product system
$T_\alpha : \Sigma \times G \to \Sigma \times G$ defined by
$T_\alpha(x,g)=(\sigma x,g\alpha(x))$, where $\alpha : \Sigma \to G$ is some continuous function depending on two co-ordinates, with a roof function $\tilde r: \Sigma \times G \to \mathbb R^{>0}$
satisfying $\tilde r(x,g)=r(x)$; to simplify notation, we will write $r$ instead of $\tilde r$.
The function $\sum_{i=1}^b \xi_i(\varphi)\omega_i(X) : M \to \mathbb R$ induces a function
$\Xi : \Sigma \to \mathbb R$, defined by
\[
\Xi(x) := \int_0^{r(x)} \left(\varphi +\sum_{i=1}^b \xi_i(\varphi)\omega_i(X) \right) \circ \pi(\sigma^t(x,0))
\, dt.
\]
It is a standard result (Proposition 6.1 of \cite{PP}) that 
\[
P\left(-P\left(\varphi + \sum_{i=1}^b \xi_i(\varphi)\omega_i(X)\right)r +\Phi +\Xi\right)=0.
\]

Now let $\oG$ be the torsion free part of the abelianization of $G$ and let $\mathfrak a : G \to \oG$ be the natural projection homomorphism. 
This gives a regular abelian cover of $M$ and transitivity of $\wtX^t : \wtM \to \wtM$
gives that $X^t : M \to M$ is $\oG$-full.
There is also a skew-product $T_{\bar \alpha} : \Sigma \times \mathbb Z^b \to
\Sigma \times \mathbb Z^b$ defined 
by
$T_{\bar \alpha}(x,m)=(\sigma x,m + \bar \alpha(x))$, where $\bar \alpha = \mathfrak a \circ \alpha$.
Both skew-products $T_\alpha$ and $T_{\bar \alpha}$ are topologically transitive 
countable state Markov shifts (where transitivity is given by Lemma 7.3 of \cite{Dougall-Sharp1}),
so we can define the \emph{Gurevi\v{c} pressure} of locally H\"older continuous
potentials, see Sarig's original paper \cite{Sarig-etds} or his survey 
\cite{Sarig-survey}.

A H\"older continuous function $F : \Sigma \to \mathbb R$ induces functions 
$\tilde F : \Sigma \times G \to \mathbb R$ and $\bar F : \Sigma \times \mathbb Z^b \to \mathbb R$
by $\tilde F(x,g) = \bar F(x,m) =F(x)$. It will not cause any confusion to denote all three functions by $F$.
For such functions, the Gurevi\v{c} pressure of $F$ with respect to $T_\alpha$ and
$T_{\bar \alpha}$, is defined by
\[
P_{\mathrm{G}}(F,T_\alpha)=
\limsup_{n \to \infty} \frac{1}{n} \log
\sum_{\substack{\sigma^nx=x \\ \alpha_n(x)=e}} e^{F^n(x)}
\]
and
\[
P_{\mathrm{G}}(F,T_{\bar \alpha})
=
\limsup_{n \to \infty} \frac{1}{n} \log
\sum_{\substack{\sigma^nx=x \\ \bar \alpha^n(x)=0}} e^{F^n(x)},
\]
respectively. (Here we have used that $T_\alpha^n(x,g) = (\sigma^nx,g\alpha^n(x))$,
where $\alpha^n(x):=\alpha(x)\alpha(\sigma x) \cdots \alpha(\sigma^{n-1}(x)$.)

The following result of \cite{Dougall-Sharp1} is key to our analysis.

\begin{proposition}[Theorem 5.1 of \cite{Dougall-Sharp1}]\label{prop:amenable_implies_shift_pressure}
If $T_\alpha$ is topologically transitive and $G$ is amenable then
\[
P_{\mathrm{G}}(F,T_\alpha) =P_{\mathrm{G}}(F,T_{\bar \alpha}).
\]
\end{proposition}

Fix $\Delta>0$ and define
\[
\overline{\mathfrak P}(\varphi) :=
\lim_{T \to \infty} \frac{1}{T} \log
\sum_{\substack{\gamma \in \mathcal P_0 \\ T < \ell(\gamma) \le T+\Delta}}
e^{\int_\gamma \varphi}.
\]
We know from Corollary 6.2 of \cite{CS} that
\[
\overline{\mathfrak P}(\varphi)=
P\left(\varphi + \sum_{i=1}^b \xi_i(\varphi)\omega_i(X)\right),
\]
so that 
\[
P\left(-\overline{\mathfrak P}(\varphi)r +\Phi +\Xi\right)=0.
\]

Now let
\[
\mathfrak P(\varphi) :=
\limsup_{T \to \infty} \frac{1}{T} \log
\sum_{\substack{\gamma \in \widetilde{\mathcal P}_0 \\ T < \ell(\gamma) \le T+\Delta}}
e^{\int_\gamma \varphi}.
\]
Clearly, $\mathfrak P(\varphi) \le \overline{\mathfrak P}(\varphi)$ and we claim that we have equality.

\begin{lemma}
We have
\[
\mathfrak P(\varphi) = \overline{\mathfrak P}(\varphi).
\]
\end{lemma}

\begin{proof}
Consider the series
\[
S_1(s) := \sum_{\substack{\gamma \in \widetilde{\mathcal P}_0}} e^{-s\ell(\gamma)+\int_\gamma \varphi}.
\]
The series $S_1(s)$ has abscissa of convergence
$\mathfrak P(\varphi)$, and so $\mathfrak P(\varphi) = \overline{\mathfrak P}(\varphi)$
if $\overline{\mathfrak P}(\varphi)$ is the abscissa of convergence of $S_1(s)$.

Now consider the corresponding series for $T_\alpha$,
\[
S_2(s) := \sum_{n=1}^\infty \frac{1}{n} \sum_{\substack{\sigma^n x=x \\\alpha^n(x)=e}}
e^{-sr^n(x) + \Phi^n(x)}.
\]
This may involve some overcounting compared to $S_1(s)$ but it is standard that 
$S_1(s)-S_2(s)$ converges for $\mathrm{Re}(s) > \overline{\mathfrak P}(\varphi) -\epsilon$, for some
$\epsilon>0$.
Thus the abscissa of convergence of $S_1(s)$ is $\overline{\mathfrak P}(\varphi)$ 
 if and only if the abscissa of convergence of
$S_2(s)$ is $\overline{\mathfrak P}(\varphi)$.

The abscissa of convergence of $S_2(s)$ is given by the value $c$ for which 
$P_{\mathrm{G}}(-cr+\Phi,T_{\alpha})=0$.
By Proposition \ref{prop:amenable_implies_shift_pressure},
\[
P_{\mathrm{G}}(-cr+\Phi,T_{\alpha}) = P_{\mathrm{G}}(-cr+\Phi,T_{\bar \alpha})
\]
and so $c= \overline{\mathfrak P}(\varphi)$, as required.
\end{proof}

To complete the proof of Proposition \ref{prop:transfer_to_abelian} we observe that the arguments in 
section 2 of the correction to \cite{Dougall-Sharp1} show that, provided $\Delta>0$ is sufficiently large,
the limsup defining $\mathfrak P(\varphi)$ is a limit.

The proof of Theorem \ref{th:amenable} now follows from Theorem \ref{th:eq-weighted-amenable}
exactly as in the first proof of Theorem \ref{abelian_livsic_theorem}.

\begin{remark}
An interesting example of an amenable cover is that associated to the second commutator
of $\pi_1(M)$. Let $M$ be a quotient $M = U/\Gamma$, where $U$ is the universal cover 
of $M$ and
$\Gamma \cong \pi_1(M)$ is a group of isometries acting freely on $U$.
The universal abelian cover of $M$ is the regular cover with covering
group $\Gamma/\Gamma' \cong H_1(M,\mathbb Z)$, where $\Gamma'=[\Gamma,\Gamma]$ is the 
commutator subgroup (derived subgroup) of $\Gamma$, generated by the set of all
commutators in $\Gamma$. The second commutator subgroup (second derived subgroup)
$\Gamma''$ is the subgroup generated by all commutators of commutators, i.e. by 
all elements of the form $[[a,b],[c,d]]$ for $a,b,c,d \in \Gamma$.
The quotient $\Gamma/\Gamma''$ is metabelian and hence amenable.
In terms of a flow on $M$, a periodic orbit has trivial Frobenius class for this cover if and only if it is 
null-homologous and lifts to a null-homologous periodic orbit on the universal abelian cover.
If $X^t : M \to M$ is a geodesic flow over a compact manifold with negative sectional curvatures
then the lifted flow on the $\Gamma/\Gamma''$ cover is topologically transitive and so
Theorem \ref{th:amenable} applies.
\end{remark}

\end{document}